\documentclass{amsart}
\usepackage{a4wide}
\usepackage{amsfonts}
\usepackage{url}
\usepackage{listings}
\theoremstyle{plain}
\newtheorem{theorem}{Theorem}
\usepackage{tcolorbox}
\newtheorem{remark}{Remark}

\newtheorem{corollary}{Corollary}

\usepackage[osf,sc]{mathpazo}

\theoremstyle{definition}

\allowdisplaybreaks

\newcommand{\bell}{\textup{B}}



\author{Sumit Kumar Jha}
\address{International Institute of Information Technology\\
Hyderabad-500032, India}
\email{kumarjha.sumit@research.iiit.ac.in} 

\begin{document}
	
\title[A formula for the number of partitions of $n$]
 {A formula for the number of partitions of $n$ in terms of the partial Bell polynomials}

\begin{abstract}
We derive a formula for $p(n)$ (the number of partitions of $n$) in terms of the partial Bell polynomials using Fa\`{a} di Bruno's formula and Euler's pentagonal number theorem.
\end{abstract}

\maketitle

\section{Main result}
Recall the classical partition function, denoted by $p(n)$, gives the number of ways of writing the integer $n$ as a sum of positive integers, where the order of summands is not considered significant. For example, $p(4)=5$, since there are $5$ ways to represent 4 as
sum of positive integers, namely, $4= 3+1=2+2=2+1+1=1+1+1+1$.\par 
We also recall another classical statistic, the {\it $(n,k)$th partial Bell polynomial} in the variables $x_{1},x_{2},\dotsc,x_{n-k+1}$, denoted by $B_{n,k}\equiv\bell_{n,k}(x_1,x_2,\dotsc,x_{n-k+1})$ (\cite[p. 134]{Comtet}, \cite[Ch. 12]{Andrews}), defined by
\begin{equation*}
\bell_{n,k}(x_1,x_2,\dotsc,x_{n-k+1})=\sum_{\substack{1\le i\le n,\ell_i\in\mathbb{N}\\ \sum_{i=1}^ni\ell_i=n\\ \sum_{i=1}^n\ell_i=k}}\frac{n!}{\prod_{i=1}^{n-k+1}\ell_i!} \prod_{i=1}^{n-k+1}\Bigl(\frac{x_i}{i!}\Bigr)^{\ell_i}.
\end{equation*}
Cvijovi\'{c} \cite{Bell} gives the following formula for calculating these polynomials
\begin{align}
\label{explicit}
 B_{n, k + 1}  =  & \frac{1}{(k+1)!} \underbrace{\sum_{\alpha_1\,=  k}^{n-1} \, \sum_{\alpha_2\,=  k-1}^{\alpha_1-1} \cdots  \sum_{\alpha_k\, = 1}^{\alpha_{k-1}-1} }_{k }
\overbrace{\binom{n}{\alpha_1}  \binom{\alpha_1}{\alpha_2} \cdots  \binom{\alpha_{k-1}}{\alpha_k}}^{k}\nonumber \times \cdots
\\
&\times x_{n-\alpha_1} x_{\alpha_1 -\alpha_2} \cdots x_{\alpha_{k-1}-\alpha_k} x_{\alpha_k} \qquad(n\geq k+1, k\,=1, 2, \ldots)
\end{align}
We prove the following here.
\begin{theorem}
We have
\begin{equation}
\label{explicit2}
p(n)=\frac{1}{n!}\sum_{k=0}^{n}(-1)^{k}\, k!\, B_{n,k}(\lambda_{1},\lambda_{2},\cdots,\lambda_{n-k+1})
\end{equation}
where
\begin{equation}
\normalfont
\label{lamb}
\lambda_{m}=
\begin{cases}
(-1)^{\frac{1+\sqrt{1+24m}}{6}}\, m! & \text{if $\frac{1+\sqrt{1+24m}}{6}\in \mathbb{Z}$,}\\
(-1)^{\frac{1-\sqrt{1+24m}}{6}}\, m! & \text{if $\frac{1-\sqrt{1+24m}}{6}\in \mathbb{Z}$,}\\
0 & \text{otherwise.}
\end{cases}
\end{equation}
\end{theorem}
\begin{proof}
We begin by the following generating function \cite[Equation 22.13]{Fine}
\begin{equation}
\sum_{n\geq 0}p(n)q^{n}=\prod_{j=1}^{\infty}\frac{1}{1-q^{j}}.
\end{equation}
We recall the Euler's pentagonal number theorem \cite[Equation 7.8]{Fine}
\begin{align}
\label{pent}
E(q)&:=\prod_{j=1}^{\infty}(1-q^{j})=\sum_{n=-\infty}^{\infty}(-1)^{n}q^{\frac{3n^{2}+n}{2}}\\
\nonumber &=1-q-q^2+q^5+q^7-q^{12}-q^{15}+q^{22}+q^{26}-\cdots.
\end{align}
Let $f(q)=1/q$. Using Fa\`{a} di Bruno's formula (\cite[p. 137]{Comtet}, \cite[Ch. 12]{Andrews}) we have
\begin{equation}
\label{faa}
{d^n \over dq^n} f(E(q)) = \sum_{k=0}^n f^{(k)}(E(q))\cdot B_{n,k}\left(E'(q),E''(q),\dots,E^{(n-k+1)}(q)\right).
\end{equation}
Since $f^{(k)}(q)=\frac{(-1)^{k}\,k!}{q^{k+1}}$ and $E(0)=1$, letting $q\rightarrow 0$ in the above equation gives
$$
p(n)\, n! = \sum_{k=0}^n (-1)^{k}\, k!\, B_{n,k}\left(E'(0),E''(0),\dots,E^{(n-k+1)}(0)\right).
$$
Then Euler's pentagonal number theorem \eqref{pent} gives us
$$
E^{(m)}(0)=\lambda_{m}
$$
where $\lambda_{m}$ is as defined in \eqref{lamb}.
\end{proof}
Combining equations \eqref{explicit} and \eqref{explicit2} we can conclude that
$$
p(n)=-\theta_{n}+\sum_{k=1}^{n-1}(-1)^{k-1} \underbrace{\sum_{\alpha_1\,=  k}^{n-1} \, \sum_{\alpha_2\,=  k-1}^{\alpha_1-1} \cdots  \sum_{\alpha_k\, = 1}^{\alpha_{k-1}-1}}_{k}\theta_{n-\alpha_1} \theta_{\alpha_1 -\alpha_2} \cdots \theta_{\alpha_{k-1}-\alpha_k} \theta_{\alpha_k}
$$
where
\begin{equation*}
\theta_{m}=
\begin{cases}
(-1)^{\frac{1+\sqrt{1+24m}}{6}} & \text{if $\frac{1+\sqrt{1+24m}}{6}\in \mathbb{Z},$}\\
(-1)^{\frac{1-\sqrt{1+24m}}{6}} & \text{if $\frac{1-\sqrt{1+24m}}{6}\in \mathbb{Z},$}\\
0 & \text{otherwise.}
\end{cases}
\end{equation*}
\begin{corollary}
Let $E(q)^{r}:=\prod_{j=1}^{\infty}(1-q^{j})^{r}=\sum_{n=0}^{\infty}p_{r}(n)q^{n}$ with $p_{r}(0)=1$ (see \cite{Atkin}). Then
$$
p(n)=\sum_{r=0}^{n}(-1)^{r}\,\binom{n+1}{r+1}\,p_{r}(n),
$$
where by the virtue of the Fa\`{a} di Bruno's formula \eqref{faa} with $f(q)=q^{l}$ we have
\begin{align*}
p_{l}(n)&=\frac{1}{n!}\sum_{k=0}^{l}\binom{l}{k}\,k!\, B_{n,k}(\lambda_{1},\ldots ,\lambda_{n-k+1}).
\end{align*}
\end{corollary} 
\begin{proof}
We start with the generating function for the partial Bell polynomials \cite[Equation (3a') on p. 133]{Comtet} as follows
\begin{align*}
{\displaystyle \sum _{n=k}^{\infty }B_{n,k}(\lambda_{1},\ldots ,\lambda_{n-k+1}){\frac {q^{n}}{n!}}}
&= {\frac {1}{k!}}\left(\sum _{j=1}^{\infty }\lambda_{j}{\frac {q^{j}}{j!}}\right)^{k} \\
&=\frac{1}{k!}(E(q)-1)^{k}\\
&=\frac{1}{k!}\sum_{r=0}^{k}(-1)^{k-r}\binom{k}{r}E(q)^{r}\\
&=\frac{1}{k!}\sum_{r=0}^{k}(-1)^{k-r}\binom{k}{r}\sum_{n=0}^{\infty}p_{r}(n)q^{n}
\end{align*}
to conclude that
$$
B_{n,k}(\lambda_{1},\ldots ,\lambda_{n-k+1})=\frac{n!}{k!}\sum_{r=0}^{k}(-1)^{k-r}\binom{k}{r}p_{r}(n).
$$
The above equation, together with the formula \eqref{explicit2}, gives us
\begin{align*}
p(n)&=\sum_{k=0}^{n}\sum_{r=0}^{k}(-1)^{r}\, \binom{k}{r}\,p_{r}(n)\\
&=\sum_{r=0}^{n}(-1)^{r}p_{r}(n)\sum_{k=r}^{n}\binom{k}{r}\\
&=\sum_{r=0}^{n}(-1)^{r}\,\binom{n+1}{r+1}\,p_{r}(n).
\end{align*}
\end{proof}
\begin{remark}
Note that similar ideas are used in \cite{Ono-etc} with relation to partition zeta functions.
\end{remark}

\end{document}